\documentclass[a4paper, 12pt, twoside]{amsart}
\usepackage{calrsfs}
\usepackage{amssymb}
\usepackage{amsthm}
\usepackage{amscd}
\usepackage{amsopn}
\usepackage{eucal}
\usepackage{etex}
\usepackage{pictex}
\usepackage[all,cmtip]{xy}
\usepackage{tikz}

\usetikzlibrary{patterns}

\newtheoremstyle{Teorema}{10pt}{10pt}{\it}{}{\bf}{. }{ }{}
\theoremstyle{Teorema}
\newtheorem{Theorem}{Theorem}[section]

\newtheorem{Corollary}[Theorem]{Corollary}

\newtheorem{Proposition}[Theorem]{Proposition}

\newtheorem{Definition}[Theorem]{Definition}
\newtheorem{Lemma}[Theorem]{Lemma}

\newtheoremstyle{PseudoHeading}{10pt}{10pt}{}{}{\bf}{. }{ }{}
\theoremstyle{PseudoHeading}

\def\Hom{\operatorname{Hom}}
\def\Aut{\operatorname{Aut}}

\def\End{\operatorname{End}}

\def\SL{\operatorname{SL}}

\def\GL{\operatorname{GL}}
\def\PSL{\operatorname{PSL}}
\def\PGL{\operatorname{PGL}}

\def\bbn{\mathbb{N}}
\def\bbz{\mathbb{Z}}
\def\bbr{\mathbb{R}}
\def\bbq{\mathbb{Q}}
\def\bbc{\mathbb{C}}

\def\bbp{\mathbb{P}}

\def\bbone{\mathbf{1}}

\def\cardop{\operatorname{card}}

\begin{document}
\title{Enumerating Trees}
\author{Robert A. Kucharczyk}
\address{Universit\"{a}t Bonn\\
Mathematisches Institut\\
Endenicher Allee 60\\
D-53115 Bonn\\
Germany}
\email{rak@math.uni-bonn.de}
\thanks{The author acknowledges financial support by the Max-Planck-Institut f\"{u}r Mathematik in Bonn}

\begin{abstract}
In this note we discuss trees similar to the Calkin-Wilf tree, a binary tree that enumerates all positive rational numbers in a simple way. The original construction of Calkin and Wilf is reformulated in a more algebraic language, and an elementary application of methods from analytic number theory gives restrictions on possible analogues.
\end{abstract}
\maketitle
\tableofcontents

\section{The Calkin-Wilf Tree}

\noindent In \cite{CalkinWilf}, Neil Calkin and Herbert Wilf introduced a remarkably beautiful\footnote{It was considered worthy by the authors of \cite{AignerZiegler09} to be included into their BOOK.} way to enumerate the positive rational numbers, drawing together several observations by Stern \cite{Stern1858} and Reznick \cite{Reznick90}. The enumeration is along a binary tree in the sense of computer science, i.e. an infinite rooted tree in which each node has two children\footnote{By the recursive procedure for constructing the tree, it seems natural to use the family metaphor in this direction. Since this is the usual terminology, we stick to it. The reverse direction would be somewhat more fitting, though, since (at least by the current state of art in reproductive medicine) everybody has precisely two parents, one of which is ``male'' and one of which is ``female''. But to produce children, you need a partner, and their number is generally not fixed to two. In either direction, an infinite chain appears problematic, although there can be little doubt that Thomas Aquinas would have preferred an infinite sequence of children.}, one of which is called ``left'' and the other ``right''. This naming should be considered not just as a device for drawing the tree, but rather as part of the mathematical structure.

Here comes its construction. The nodes of the tree are labelled by positive rational numbers. For ease of notation, we write each such number in the form $\frac{p}{q}$ with $p,q\in\bbn\smallsetminus \{ 0\}$ coprime. The rule for labelling is recursive: the tree's root is labelled by $\frac{1}{1}$. If a node is labelled $\frac{p}{q}$, then its left child bears the label $\frac{p}{p+q}$ and its right child bears the label $\frac{p+q}{q}$. By induction we directly see that these are reduced fractions as written.

Before proving and stating the basic properties of this tree, we encourage the reader to contemplate Table \ref{CWBild} where the first few layers are shown.

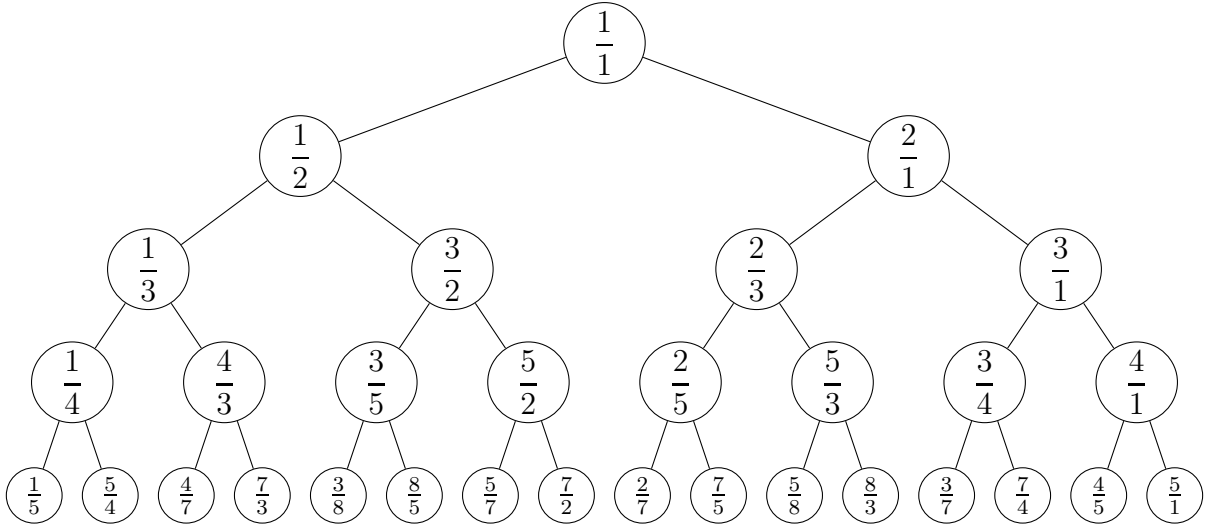
\begin{table}\label{CWBild}
\begin{center}
\caption{The first five layers of the Calkin-Wilf tree}
\begin{tikzpicture}
[every node/.style={draw,circle,inner sep=2pt},
level 1/.style={sibling distance=80mm},
level 2/.style={sibling distance=40mm},
level 3/.style={sibling distance=20mm},
level 4/.style={sibling distance=10mm}]
\node {$\dfrac{1}{1}$}
	child {node {$\dfrac{1}{2}$}
		child {node {$\dfrac{1}{3}$}
			child {node {$\dfrac{1}{4}$}
				child {node {$\frac{1}{5}$}}
				child {node {$\frac{5}{4}$}}
			}
			child {node {$\dfrac{4}{3}$}
				child {node {$\frac{4}{7}$}}
				child {node {$\frac{7}{3}$}}
			}
		}
		child {node {$\dfrac{3}{2}$}
			child {node {$\dfrac{3}{5}$}
				child {node {$\frac{3}{8}$}}
				child {node {$\frac{8}{5}$}}
			}
			child {node {$\dfrac{5}{2}$}
				child {node {$\frac{5}{7}$}}
				child {node {$\frac{7}{2}$}}
			}
		}
	}
	child {node {$\dfrac{2}{1}$}
		child {node {$\dfrac{2}{3}$}
			child {node {$\dfrac{2}{5}$}
				child {node {$\frac{2}{7}$}}
				child {node {$\frac{7}{5}$}}
			}
			child {node {$\dfrac{5}{3}$}
				child {node {$\frac{5}{8}$}}
				child {node {$\frac{8}{3}$}}
			}
		}
		child {node {$\dfrac{3}{1}$}
			child {node {$\dfrac{3}{4}$}
				child {node {$\frac{3}{7}$}}
				child {node {$\frac{7}{4}$}}
			}
			child {node {$\dfrac{4}{1}$}
				child {node {$\frac{4}{5}$}}
				child {node {$\frac{5}{1}$}}
			}
		}
	};
\end{tikzpicture}
\end{center}
\end{table}

\begin{Proposition}[Calkin-Wilf]\label{CWProposition}
In the Calkin-Wilf tree, every positive rational appears exactly once.
\end{Proposition}
\begin{proof}
For ease of parlance, we confuse nodes with their labels.

Writing a positive rational as $p/q$ with $p,q$ coprime positive integers, we proceed by induction on $m=\max (p,q)$. For $m=1$ there is only $p=q=1$ to consider. The rational number $1/1=1$ does appear in the tree, namely at its root; it cannot occur anywhere else, since each left child $p/(p+q)$ is smaller than $1$ and each right child $(p+q)/q$ is bigger than $1$.

Assume now that the statement is proved for all $m<m_0$, and let $x=p/q$ with $\max (p,q)=m_0$. Then either $x<1$ or $x>1$. In the first case, we have $m_0=q>p$, hence $x$ is the left child of the (by assumption) unique node labelled $p/(q-p)$, and since it cannot be a right child (else $x>1$), it cannot occur at any other place. Similarly, if $x>1$, it must be a right child, and it must be the right child of $(p-q)/q$ which, by assumption, does occur exactly once.
\end{proof}
The proof already shows that the position of a positive rational $p/q$ can be determined by performing the Euclidean algorithm on $p$ and $q$. It is also clear that the continued fraction expansion of $p/q$ and the sequence of left / right moves one has to make from $1$ in order to get to $p/q$ are easily translated into one another.

So, if we write down the first line, then the second line, then the third line of the Calkin-Wilf tree and so on, we obtain a list of the positive rationals in which each of them appears exactly once, i.e. a bijection $\bbn_0\to\bbq_{>0}$. As can be checked from Table \ref{CWBild}, this list begins with
\begin{equation}\label{CWList}
\frac{1}{1},\frac{1}{2},\frac{2}{1},\frac{1}{3},\frac{3}{2},\frac{2}{3},\frac{3}{1},\frac{1}{4},\frac{4}{3},\frac{3}{5},\frac{5}{2},\frac{2}{5},\frac{5}{3},\frac{3}{4},\frac{4}{1},\frac{1}{5},\frac{5}{4},\frac{4}{7},\frac{7}{3},\frac{3}{8},\frac{8}{5},\frac{5}{7},\frac{7}{2},\frac{2}{7},\frac{7}{5},\frac{5}{8},\frac{8}{3},\frac{3}{7},\frac{7}{4},\frac{4}{5},\frac{5}{1},\ldots
\end{equation}
The attentive reader will long have noticed that the denominator of each term is equal to the numerator of its successor. This can easily be proved by induction. Hence there must be a function $f :\bbn_0\to\bbn$ such that $f(n)$ and $f(n+1)$ are coprime, and the $n$-th element of the sequence (\ref{CWList}) is equal to $f(n)/f(n+1)$. It is proved in \cite{CalkinWilf} that $f(n)$ is the number of ways to partition $n$ into powers of two, each power occurring at most twice.

Moshe Newman also has found a simple recursive construction of the sequence (\ref{CWList}) that does not make reference to the tree anymore: it is the sequence $(a_n)$ with $a_0=1$ and
\begin{equation*}
a_{n+1}=\frac{1}{1+\lfloor a_n\rfloor -\{ a_n\} }.
\end{equation*}
Here, $\lfloor a_n\rfloor$ is the largest integer $\le a_n$ and $\{ a_n\} =a_n-\lfloor a_n\rfloor$ is the ``fractional part'' of $a_n$. This was a solution to a problem raised by Donald Knuth in the American Mathematical Monthly, see \cite{Knuth}.

For more details, and further interesting developments in directions not touched upon in this paper, see \cite{BatesMansour}, \cite{BergstraTucker}, and \cite{Northshield}.

We wish to look upon the Calkin-Wilf tree from another point of view: that of M\"{o}bius transformations. Recall that the group of M\"{o}bius transformations over a field $K$ is the group $\PGL_2(K)=\GL_2(K)/K^{\times }$. We introduce the following notation:
$$\begin{bmatrix}
a&b\\
c&d
\end{bmatrix}$$
is the element of $\PGL_2(K)$ represented by
$$\begin{pmatrix}
a&b\\
c&d
\end{pmatrix}\in\GL_2(K).$$
These M\"{o}bius transformations operate upon $\bbp^1(K) =K\cup\{\infty\}$ in the well-known way
$$\begin{bmatrix}
a&b\\
c&d
\end{bmatrix}\cdot z=\frac{az+b}{cz+d}.$$
The subgroup $\PSL_2(\bbz )=\SL_2(\bbz )/\{\pm\bbone\}$ of $\PGL_2(\bbq )$ has been much investigated, and it operates transitively on $\bbp^1(\bbq )$. A closer look at the rules generating the Calkin-Wilf tree shows that if a node is labelled by $x\in\bbq_{>0}\subset\bbp^1(\bbq )$, then its left child is labelled by $L(x)$ and its right child by $R(x)$, where
\begin{equation}
L=\begin{bmatrix}1&0\\ 1&1\end{bmatrix}\text{ and }R=\begin{bmatrix}1&1\\ 0&1\end{bmatrix}.
\end{equation}
These choices may at first glance look arbitrary, but we shall argue in the next section that they are not.

\section{The Monoid $\SL_2(\bbn_0)$}

\noindent Most of the literature on M\"{o}bius transformations deals with groups of them, but here we shall be concerned with monoids. Since this term is somewhat ambigous, let us fix a definition:
\begin{Definition}
A \emph{monoid} is a set $M$ together with a binary operation $\cdot :M\times M\to M$ with the following properties:
\begin{enumerate}
\item it is associative, i.e. $x(yz)=(xy)z$ for any $x,y,z\in S$, and
\item there exists an identity element, i.e. an element $e\in M$ such that $ex=x=xe$ for all $x\in M$.
\end{enumerate}
\end{Definition}
Such an identity element is necessarily unique.

As usual in algebra, one can now introduce free monoids. If $A$ is a set (considered as an ``alphabet''), then the free monoid\footnote{Friends of abstract nonsense will immediately recognize that this is equivalent to the definition in terms of an adjoint functor to the forgetful functor to sets that they sure would have proposed.} $\mathrsfs{F}(A)$ generated by $A$ consists of all formal words of finite length in the alphabet $A$. Multiplication is given by concatenation. The empty word $\varnothing$ is allowed and serves as the identity element in $\mathrsfs{F}(A)$.

If $M$ is a monoid and $A\subseteq M$ a subset, we say that $M$ is \emph{free on} $A$ or \emph{freely generated by} $A$ if the obvious map $\mathrsfs{F}(A)\to M$ is an isomorphism of monoids; in other words, if each element of $M$ can be written in a unique way as a product of elements of $A$.

What do free monoids look like? Certainly the free monoid on one element is isomorphic to $\bbn_0$ with addition. The free monoid on two generators is much richer in structure. It is tempting to think of it as similar to the free group on two generators; but it is in fact much more rigid. Namely:
\begin{Lemma}\label{AutOfFreeMonoids}
Let $X=\{ x_1,\ldots ,x_n\}$ be a finite set with $n$ elements, and set $\mathrsfs{F}_n=\mathrsfs{F}(X)$. Then any automorphism of $\mathrsfs{F}_n$ is obtained from a permutation of the $x_i$.
\end{Lemma}
\begin{proof}
Consider $X$ as a subset of $\mathrsfs{F}_n$. Then an element $\gamma\in\mathrsfs{F}_n$ is in $X$ if and only if $\gamma\neq 1$ and whenever $\gamma =\delta\varepsilon$, then at least one of $\delta$, $\varepsilon$ is equal to $1$. Hence any automorphism of $\mathrsfs{F}_n$ takes $X$ to itself.

In particular, $\mathrsfs{F}_n\simeq\mathrsfs{F}_m$ if and only $m=n$.
\end{proof}
An automorphism of $\mathrsfs{F}_n$ is of course determined by what it does on $X$, and so we get an automorphism $\Aut\mathrsfs{F}_n\simeq\mathfrak{S}_n$, the symmetric group. By contrast, if $F_n$ denotes the free \emph{group} on $n$ letters, the automorphism group $\Aut F_n$ is huge. But the picture becomes clearer when one notices that the analogue of $\Aut F_n$ should not be the group $\Aut\mathrsfs{F}_n$, but the monoid $\End\mathrsfs{F}_n$, which is much larger.

But now enough abstract algebra; we finally introduce the object announced in the section title. As one would expect from the notation, the monoid $\SL_2(\bbn_0)$ consists of all $(2\times 2)$-matrices with entries in $\bbn_0$ having determinant one, with matrix multiplication as the monoid operation. In other words, $\SL_2(\bbn_0)$ is the sub-monoid of $\SL_2(\bbz )$ consisting of all matrices with nonnegative entries. Note that the composition
\begin{equation}
\SL_2(\bbn_0)\to\SL_2(\bbz )\to\PSL_2(\bbz )
\end{equation}
is injective, so that we can and will view $\SL_2(\bbn_0)$ as a submonoid of $\PSL_2(\bbz )$. Hence the M\"{o}bius transformations $L$ and $R$ introduced above can be viewed as elements of $\SL_2(\bbn_0)$.
\begin{Proposition}[Folklore]\label{SLZwoNFrei}
The monoid $\SL_2(\bbn_0)$ is freely generated by the elements
\begin{equation}
L=\begin{pmatrix}1&0\\ 1&1\end{pmatrix}\text{ and }R=\begin{pmatrix}1&1\\ 0&1\end{pmatrix}.
\end{equation}
\end{Proposition}
\begin{proof}
We first show that $\SL_2(\bbn_0)$ is generated by $L$ and $R$. So let
$$\gamma =\begin{pmatrix}
a&b\\
c&d
\end{pmatrix}
\in\SL_2(\bbn_0).$$ We set $\Sigma (\gamma )=a+b+c+d$ and proceed by induction on $\Sigma (\gamma )$. It is clear that $\Sigma (\gamma )\ge 2$, with equality if and only if $\gamma =\bbone$. Hence we may assume that $\Sigma (\gamma )\ge 3$ and $\gamma\neq\bbone$. Consider the two products in $\SL_2(\bbz )$:
$$L^{-1}\gamma =\begin{pmatrix}
a&b\\
c-a&d-b
\end{pmatrix}\text{ and }R^{-1}\gamma=
\begin{pmatrix}
a-c&b-d\\
c&d
\end{pmatrix}.$$
By Lemma \ref{LemmaInequality} below, $(a-c)(b-d)\ge  0$, so at least one of these is in $\SL_2(\bbn_0)$. For sake of simplicity, assume that $L^{-1}\gamma\in\SL_2(\bbn_0)$, the other cases is treated analogously. Then $\Sigma (L^{-1}\gamma )<\Sigma (\gamma )$, so by induction hypothesis $L^{-1}\gamma$ is a product of $L$ and $R$. Hence so is $\gamma$.

Now we have proved that $L$ and $R$ generate $\SL_2(\bbn_0)$. As to freedom, we show that $\SL_2(\bbn_0)$ is the disjoint union of the sets $\{\bbone\}$, $L\cdot\SL_2(\bbn_0)$ and $R\cdot\SL_2(\bbn_0)$. That it is their union follows from the fact already proved (that $L$ and $R$ generate $\SL_2(\bbn_0)$), and the disjointness follows by contemplating the equations
$$
L\cdot\begin{pmatrix}
a&b\\
c&d
\end{pmatrix}=\begin{pmatrix}
a&b\\
a+c&b+d
\end{pmatrix}\text{ and }R\cdot\begin{pmatrix}
a&b\\
c&d
\end{pmatrix}=\begin{pmatrix}
a+c&b+d\\
c&d
\end{pmatrix}.
$$
(Just consider the possible order relations between entries.) But this observation gives an induction proof on word length for the uniqueness of a word defining an element.
\end{proof}
We should remark that $L$ and $R$ do \emph{not} generate a free group of matrices, nor of M\"{o}bius transformations. To be more specific, the subgroup of $\GL_2(\bbq )$ they generate is $\SL_2(\bbz )$, and correspondingly the subgroup of $\PGL_2(\bbq )$ they generate is $\PSL_2(\bbz )$. Both groups are well-known to contain nontrivial torsion elements. For instance, we have the equations $(RL^{-1}R)^2=\bbone$ in $\PSL_2(\bbz )$ and $(RL^{-1}R)^4=\bbone$ in $\SL_2(\bbz )$.
\begin{Lemma}\label{LemmaInequality}
Let 
$$\begin{pmatrix}
a&b\\
c&d\end{pmatrix}\in\SL_2(\bbn_0)$$
be different from the identity matrix. Then $(a-c)(b-d)\ge 0$.
\end{Lemma}
\begin{proof}
Assume that $(a-c)(b-d)<0$, i.e. that $a-c$ and $b-d$ are both nonzero and have opposite signs. There are two cases.

The first case is that $a>c$ and $d>b$. Then $a\ge c+1$ and $d\ge b+1$, whence
$$1=ad-bc\ge (c+1)(b+1)-bc =b+c+1\ge 1,$$
so equality has to hold everywhere, and $b=c=0$. From $ad-bc=1$ we get that $a=d=1$, hence the matrix in question is the identity matrix.

The second case is that $c>a$ and $b>d$. Then $c\ge a+1$ and $b\ge d+1$, so that
$$-1=bc-ad \ge (a+1)(d+1)-ad =a+d+1\ge 1,$$
contradiction.
\end{proof}
We can now reinterpret the Calkin-Wilf tree in a new light: it is the directed Cayley graph of $\SL_2(\bbn_0)$. Let us make this precise.
\begin{Definition}
A \emph{directed graph} is a quadruple $(V,E,s,t)$, where $V$ and $E$ are sets (of ``vertices'' and ``edges'', respectively) and $s$ and $t$ are maps $E\to V$ (designating ``source'' and ``target'').
\end{Definition}
When we draw (or imagine) a directed graph, we draw a node for each $v\in V$, and for each $e\in E$ an arrow originating in $s(e)$ and ending in $t(e)$. Forgetting the orientations of the arrows gives a graph in the usual sense, and we say that a directed graph is a (directed) tree if this underlying undirected graph is a tree.
\begin{Definition}
Let $M$ be a monoid and $A\subseteq M$ a generating set. The \emph{directed Cayley graph} $C(M,A)$ is the directed graph $(V,E,s,t)$ with $V=M$ and $E=M\times A$, such that $s(\mu ,\alpha )=\mu$ and $t(\mu ,\alpha )=\alpha\mu$.
\end{Definition}
In less formal terms, the vertices are in bijection with $M$, and for each $\mu\in M$ and each $\alpha\in A$ we draw an arrow from $\mu$ to $\alpha\mu$.

Note that if $A$ freely generates $M$, then $C(M,A)$ is a directed tree where every arrow points away from the ``root'' $e\in M$.

When treating Cayley graphs of groups, there is often a nasty ambiguity involved in choosing a set of generators. As a consequence, one is mainly interested in properties of the Cayley graph that do not depend on the choice of a particular set of generators. Here, however, we are in a much nicer situation. Proposition \ref{SLZwoNFrei} gives us an explicit isomorphism between $\mathrsfs{F}_2$ and $\SL_2(\bbn_0)$, and from Lemma \ref{AutOfFreeMonoids} we learn that $\{ L,R\}$ is \emph{the only} subset that freely generates $\SL_2(\bbn_0)$. In other words, if we want a tree, we have no other choice for our generators.
\begin{Proposition}\label{BijektionSLZwoQ}
Consider $\SL_2(\bbn_0)$ as a submonoid of the group $\PSL_2(\bbz )$, acting on $\bbp^1(\bbq )$ by M\"{o}bius transformations. The orbit map $\gamma\mapsto \gamma (1)$ defines a bijection $\Omega :\SL_2(\bbn_0)\to\bbq_{>0}$.

Furthermore, $\Omega$ defines an isomorphism of directed graphs between the directed Cayley tree $C(\SL_2(\bbn_0),\{ L,R\} )$ and the Calkin-Wilf tree. Here we identify the vertex set of the Calkin-Wilf tree with $\bbq_{>0}$, and we orient each of its edges as pointing away from $1$.\hfill $\square$
\end{Proposition}
This has an amusing simple consequence in terms of Diophantine equations:
\begin{Corollary}
Let $p,q$ be coprime positive integers. Then there exist unique $a,b,c,d\in\bbn_0$ with $a+b=p$, $c+d=q$ and $ad-bc=1$.
\end{Corollary}
\begin{proof}
Set $x=p/q$. The system of equations given above can be translated into $\gamma (1)=x$ for $\gamma\in\SL_2(\bbn_0)$.
\end{proof}

\section{Injective Families}

\noindent We are looking for generalisations of the Calkin-Wilf tree; we first generalise the original construction in four different respects and then ask ourselves if we get any new examples with comparably nice properties.
\begin{enumerate}
\item Replace $2$ by any positive integer $n$: consider directed trees in which every node has $n$ (ordered) children.
\item Replace $\bbq$ by any number field.
\item Replace the initial value $1\in\bbp^1(\bbq )$ by any $x_0\in\bbp^1(K)$.
\item Replace the two M\"{o}bius transformations $L$ and $R$ by $n$ rational maps $f_1,\ldots ,f_r\in K(t)$.
\end{enumerate}
These data (i) --- (iv) should fit together in the following way: if we label the tree in (i) in such a way that the root is labelled $x_0$, and that if a node is labelled by $x\in\bbp^1(K)$, then its $n$ children are labelled $f_1(x),\ldots ,f_n(x)$, in this order. Then every element $x\in\bbp^1(K)$ should appear at most once in the tree, and the set of those that do occur should be some ``simple'' subset of $\bbp^1(K)$ (in the Calkin-Wilf tree, it would be $\bbq_{>0}$ which is arguably quite simple). Of course, what we mean by ``simple'' has to become clear in the course of the discussion.

Let us first consider the tree. The description can be made more conceptual by saying that it should be the Cayley tree $C(\mathrsfs{F}(X),X)$, where $X=\{ x_1,\ldots ,x_n\}$ with the $x_i$ pairwise distinct. As above, we set $\mathrsfs{F}_n=\mathrsfs{F}(X)$, and in addition $C(\mathrsfs{F}_n)$ short for $C(\mathrsfs{F}(X),X)$.

Our rational maps should, of course, be nonconstant; hence they should live in the monoid $\mathrsfs{R}(K)$ which consists of all nonconstant rational maps $f\in K(t)$, with composition $f\circ g$ as multiplication. This may be viewed as a sub-monoid of the monoid of endomorphisms $\End\bbp_K^1=\Hom_K(\bbp^1_K,\bbp^1_K)$. Here $\bbp^1_K$ is considered as a $K$-variety. The invertible elements in this monoid are precisely the M\"{o}bius transformations, so that we get a canonical identification $\mathrsfs{R}(K)^{\times }=\PGL_2(K)$. Note that, since $K$ is infinite, we need not distinguish between a rational function as a formal expression and the map $\bbp^1(K)\to\bbp^1(K)$ it induces.
\begin{Proposition}
For every number field $K$, the monoid $\mathrsfs{R}(K)$ is infinitely generated.
\end{Proposition}
\begin{proof}
First we show that certain groups are not finitely generated. To begin with, an abelian $2$-torsion group is the same as an $\mathbb{F}_2$-vector space; hence such an abelian group is finitely generated if and only if it is finite. For any number field, the group $K^{\times }/(K^{\times })^2$ is infinite\footnote{This can be seen, for instance, as follows: By Dirichlet's density theorem, see \cite[Chapter VII, Theorem 13.2]{Neukirch99}, there are infinitely many prime ideals in the ring of integers $\mathfrak{o}_K$ which are principal ideals. Let these be $\mathfrak{p}_1,\mathfrak{p}_2$ etc., and let $p_k$ be a generator of $\mathfrak{p}_k$. Then the elements $p_1,p_2$ etc. are all distinct modulo $(K^{\times})^2$.}. Hence it is infinitely generated, and therefore also the group $\PGL_2(K)$, which surjects onto it, must be infinitely generated.

But from this it follows that $\mathrsfs{R}(K)$ cannot be finitely generated. Suppose it were, say generated by $f_1,\ldots ,f_r,g_1,\ldots ,g_s$ with $\deg f_i=1$ and $\deg g_i>1$. Since $\deg (\varphi\circ\psi )=\deg\varphi\cdot\deg\psi$, we see that any composition containing at least one $g_i$ must have degree $>1$. So the monoid (and hence also the group) $\PGL_2(K)$ must be generated by $f_1,\ldots ,f_r$, which we have just seen to be impossible.
\end{proof}
It is all the more astonishing that we can express all $f\in\mathrsfs{R}(K)$ as compositions of just two admittedly strange maps $\bbp^1(K)\to\bbp^1(K)$.
\begin{Theorem}[Sierpi\'{n}ski]\label{SierpinskiTheorem}
Let $A$ be an infinite set, and let $\mathrsfs{M}(A)$ be the monoid of \emph{all} maps $A\to A$, with composition of maps as monoid composition. Let $X\subset\mathrsfs{M}(A)$ be any countable subset. Then there exist elements $\varphi ,\psi\in\mathrsfs{M}(A)$ such that $X$ is contained in the submonoid of $\mathrsfs{M}(A)$ generated by $\varphi$ and $\psi$.\hfill $\square$
\end{Theorem}
This Theorem was first proved in \cite{Sierpinski35}; shortly afterwards, Banach gave a very elegant proof, see \cite{Banach35}.
\begin{Corollary}
For any countable field $K$, there exist two maps $\varphi ,\psi$ from $\bbp^1(K)=K\cup\{\infty \}$ to itself such that \emph{every nonconstant rational map} $\bbp^1(K)\to\bbp^1(K)$ can be written as a finite composition involving only $\varphi$ and $\psi$.
\end{Corollary}
\begin{proof}
Apply Theorem \ref{SierpinskiTheorem} to $A=\bbp^1(K)$ and $X=\mathrsfs{R}(K)$.
\end{proof}

So having chosen rational maps $f_1,\ldots ,f_n$, we consider the unique morphism of monoids $h:\mathrsfs{F}_n\to\mathrsfs{R}(K)$ with $h(x_i)=f_i$; then our tree is the Cayley tree $C(\mathrsfs{F}_n)$, where the node corresponding to $\gamma\in\mathrsfs{F}_n$ is labelled by $h(\gamma )(x_0)$. This defines an ``evaluation'' map
\begin{equation}\label{Evaluation}
\Omega :\mathrsfs{F}_n\to\bbp^1(K),\quad\gamma\mapsto h(\gamma )(x_0).
\end{equation}
\begin{Definition}
Let $K$ be a number field, let $x_0\in\bbp^1(K)$ and let $f_1,\ldots ,f_n\in\mathrsfs{R}(K)$. The family $(f_1,\ldots ,f_n)\in\mathrsfs{R}(K)^n$ is called \emph{injective at $x_0$} if the map $\Omega$ as in (\ref{Evaluation}) is injective.
\end{Definition}
Clearly, a family $(f_1,\ldots ,f_n)\in\mathrsfs{R}(K)^n$ is injective at $x_0$ if and only if the $f_i$ generate a free submonoid $\Gamma\subset\mathrsfs{R}(K)$ and the orbit map $\Gamma\to\bbp^1(K)$ sending $\gamma$ to $\gamma (x_0)$ is injective. By conjugating with a suitable M\"{o}bius transformation, we can always assume that $x_0=1$.

Some interesting injective families over $\bbq$, all of whose members are M\"{o}bius transformations, have been found by S.H. Chan, see \cite{Chan2011}. These give rather forests with a finite number of components, instead of isolated trees. For the reader's convenience, we describe them in our terms.

For every integer $k\ge 2$, a family $\mathrsfs{G}_k$ is defined by consisting of these $2k$ M\"{o}bius transformations:
\begin{equation*}
\begin{bmatrix}
1&0\\
2&1
\end{bmatrix},\begin{bmatrix}
2&1\\
3&2
\end{bmatrix},\ldots ,\begin{bmatrix}
k-1&k-2\\
k&k-1
\end{bmatrix}, \begin{bmatrix}
k&k-1\\
k&k
\end{bmatrix},
\end{equation*}
\begin{equation*}
\begin{bmatrix}
k&k\\
k-1&k
\end{bmatrix},\begin{bmatrix}
k-1&k\\
k-2&k-1
\end{bmatrix},\ldots ,
\begin{bmatrix}
2&3\\
1&2
\end{bmatrix},\begin{bmatrix}
1&2\\
0&1
\end{bmatrix}.
\end{equation*}
It is injective on each of the initial values $x_1,\ldots ,x_{2k-1}$ given by
\begin{equation*}
\frac{1}{2},\frac{2}{3},\ldots ,\frac{k-1}{k},\frac{k}{k},\frac{k}{k-1},\ldots ,\frac{3}{2},\frac{2}{1}.
\end{equation*}
Furthermore, the orbits $\Gamma (x_1),\ldots ,\Gamma (x_{2k-1})$ are disjoint and their union is $\bbq_{>0}$. All this is proved in \cite[Theorem 4]{Chan2011}.

There is a similar infinite family of injective families; they enumerate the slightly more complicated set $\bbq_{>0}^{\text{even}}$ of all positive rational numbers $\frac{p}{q}$ with $p,q$ coprime and $pq$ even. For every integer $k\ge 1$, let $\mathrsfs{H}_k$ be the family of $2k+1$ M\"{o}bius transformations:
\begin{equation*}
\begin{bmatrix}
1&0\\
2&1
\end{bmatrix},\begin{bmatrix}
2&1\\
3&2
\end{bmatrix},\ldots ,\begin{bmatrix}
k&k-1\\
k+1&k
\end{bmatrix},\begin{bmatrix}
k+1&k\\
k&k+1
\end{bmatrix},
\end{equation*}
\begin{equation*}
\begin{bmatrix}
k&k-1\\
k+1&k
\end{bmatrix},\ldots ,\begin{bmatrix}
2&3\\
1&2
\end{bmatrix},
\begin{bmatrix}
1&2\\
0&1
\end{bmatrix}.
\end{equation*}
It is injective on each of the initial values $y_1,\ldots ,y_{2k}$ given as
\begin{equation*}
\frac{1}{2},\frac{2}{3},\ldots ,\frac{k}{k+1},\frac{k+1}{k},\ldots ,\frac{3}{2},\frac{2}{1}.
\end{equation*}
The orbits $\Gamma (y_1),\ldots ,\Gamma (y_{2k})$ are disjoint and their union is $\bbq_{>0}^{\text{even}}$. This can be found in \cite[Theorems 2 and 5]{Chan2011}. Theorem 2 in op. cit. is followed by a detailed discussion of the simplest case $k=1$.

Similar to the interpretation of the denominators and numerators of the Calkin-Wilf sequence as a combinatorial function, there are further combinatorial interpretations of these forests in \cite{Chan2011}.
 
\section{Heights on $\bbp^1$ and the Distribution of Points}

\noindent Let $K$ be a number field. A \emph{place} of $K$ is an equivalence class of valuations; denote the set of all places of $K$ by $\mathrsfs{P}(K)$. If $\mathfrak{p}$ is a place of $K$, write $K_{\mathfrak{p}}$ for the corresponding completion. For every place $\mathfrak{p}$ we choose a representing valuation $|\cdot |_{\mathfrak{p}}: K\to [0,\infty )$ in the following way:
\begin{enumerate}
\item If $\mathfrak{p}$ is real, there is a unique isomorphism of fields $K_{\mathfrak{p}}\simeq\bbr$, and we pull back along this isomorphism the usual absolute value $|x|=\max (x,-x)$ on the reals.
\item If $\mathfrak{p}$ is complex, there are two isomorphisms $\tau, \overline{\tau }:K_{\mathfrak{p}}\simeq\bbc$ of \emph{topological} fields, and we set $|x|_{\mathfrak{p}}=\tau (x)\overline{\tau }(x)$.
\item If $\mathfrak{p}$ is non-archimedean, let $q$ be the cardinality of the corresponding residue class field. Let $\pi\in K$ be a uniformising element; we normalise $|\cdot |_{\mathfrak{p}}$ in such a way that $|\pi |_{\mathfrak{p}}=\frac{1}{q}$.
\end{enumerate}
With these normalisations, we have the famous product formula, see \cite[Chapter III, Proposition 1.3]{Neukirch99}: for any $x\in K^{\times }$, all but a finite number of the $|x|_{\mathfrak{p}}$ are equal to $1$, and
\begin{equation}
\prod_{\mathfrak{p}\in\mathrsfs{P}(K)}|x|_{\mathfrak{p}}=1.
\end{equation}
As a consequence, the following construction gives a well-defined function on $\bbp^n(K)$ which can be thought of as measuring the arithmetic complexity of a point.
\begin{Definition}
Let $K$ be a number field of degree $d$ and let $x\in\bbp^n(K)$. Choose $x_0,\ldots ,x_n\in K$ such that $x=(x_0:\cdots :x_n)$; the \emph{(absolute) height} of $x$ is the real number
\begin{equation}
H(x)=\sqrt[d]{\prod_{\mathfrak{p}\in\mathrsfs{P}(K)}\max (|x_0|_{\mathfrak{p}},\ldots ,|x_n|_{\mathfrak{p}})}.
\end{equation}
The \emph{(absolute) logarithmic height} of $x$ is the real number
\begin{equation}
h(x)=\log H(x).
\end{equation}
\end{Definition}
We always have $H(x)\ge 1$ and therefore $h(x)\ge 0$, with equality if and only if $x$ is a root of unity, see \cite[Theorem 3.8]{Silverman07}. The absolute height is defined in such a way that the functions $H\colon\bbp^1(K)\to [1,\infty )$ for varying $K$ glue together to $H\colon\bbp^n(\overline{\bbq })\to [1,\infty )$, similarly for $h$.

For $K=\bbq$, there is a description of the height which is much more intuitive and makes computations much easier: if $x\in\bbp^n(\bbq )$, we can write it as $x=(x_0:\cdots :x_n)$ with $x_0,\ldots ,x_n\in\bbz$ coprime. Then
\begin{equation}
H(x)=\max (|x_0|_{\infty },\ldots ,|x_n|_{\infty }).
\end{equation}
Here, of course, $|\cdot |_{\infty }$ is the usual absolute value on $\bbz\subset\bbr$, i.e. $|a|_{\infty }=\max (a,-a)$.

We now examine how $H(f(x))$ relates to $H(x)$, where $f$ is a rational function. First we consider the case of M\"{o}bius transforms. By identifying the matrix entries with coordinates, we can view $\GL_2(K)$ as a subset of $K^4$. This is compatible with the action of $K^{\times }$, on the matrix group by multiplication with scalar matrices, and on the linear space by multiplication with scalars. So we can view $\PGL_2(K)=\GL_2(K)/K^{\times }$ as a subset of $\bbp^3(K)$ and define the height of an element of $\PGL_2(K)$ as the height of the corresponding point in $\bbp^3(K)$. By the simple description of heights for $K=\bbq$, we get an equally simple description of the height of an element $\gamma\in\PGL_2(\bbq )$: represent $\gamma$ by a matrix
\begin{equation*}
\begin{pmatrix}
a&b\\
c&d
\end{pmatrix}\in\GL_2(\bbq )
\end{equation*}
with $a,b,c,d\in\bbz$ having greatest common divisor $1$. Then
$$H(\gamma )=H((a:b:c:d))=\max (|a|_{\infty },|b|_{\infty },|c|_{\infty },|d|_{\infty }).$$
\begin{Lemma}\label{HGammaGleichHGammaInvers}
Let $K$ be a number field of degree $d$ and $\gamma ,\delta\in\PGL_2( K)$. Then $H(\gamma )=H(\gamma^{-1})$.
\end{Lemma}
\begin{proof}
If $\gamma\in\PGL_2(K)$ is represented by the matrix $A$, then $\gamma^{-1}$ is represented by the matrix $A^{-1}=(\det A)^{-1}A^{\sharp }$, where the matrix $A^{\sharp }$ is obtained by permuting the entries of $A$ in a well-known fashion and multiplying two of them with $-1$. But by the definition of $\PGL_2$, we see that $\gamma^{-1}$ is also represented by $A^{\sharp }$, whence $H(\gamma )=H(\gamma^{-1})$.
\end{proof}
\begin{Proposition}\label{DistortionOfHeightByMoebius}
Let $K$ be a number field of degree $d$, let $x\in\bbp^1(K)$ and $\gamma\in\PGL_2(K)$. Then
$$\frac{1}{2H(\gamma )}H(x)\le H(\gamma (x))\le 2H(\gamma )H(x).$$
\end{Proposition}
\begin{proof}
We only need to show the second inequality; the first will follow by replacing $\gamma$ by $\gamma^{-1}$ and using Lemma \ref{HGammaGleichHGammaInvers}. So choose a representative matrix
$$\begin{pmatrix}
a&b\\
c&d
\end{pmatrix}\in\GL_2(K)$$
for $\gamma$. Write $x=(x_0:x_1)$. Then for any place $\mathfrak{p}$ of $K$ we get
$$\max (|ax_0+bx_1|_{\mathfrak{p}},|cx_0+dx_1|_{\mathfrak{p}})\le t_{\mathfrak{p}}\cdot \max (|a|_{\mathfrak{p}},|b|_{\mathfrak{p}},|b|_{\mathfrak{p}},|d|_{\mathfrak{p}})\cdot\max (|x_0|_{\mathfrak{p}},|x_1|_{\mathfrak{p}})$$
by the triangle inequality; here $t_{\mathfrak{p}}$ is $1$ if $\mathfrak{p}$ is non-archimedean, $2$ if $\mathfrak{p}$ is real and $4$ if $\mathfrak{p}$ is complex. Taking the product over all $\mathfrak{p}$ and then taking $d$-th roots yields the desired result.
\end{proof}
Thus M\"{o}bius transformations can only change the height by a multiplicative factor. With some more effort, one obtains the following special case of \cite[Theorem 3.11]{Silverman07}:
\begin{Theorem}\label{TheoremSilverman}
Let $K$ be a number field and $f\in\mathrsfs{R}(K)$ a rational map of degree $d$. Then there exist constants $c_1,c_2>0$ such that for all $x\in\bbp^1(K)$,
\begin{equation*}
c_1\cdot H(x)^d\le H(f(x))\le c_2\cdot H(x)^d.
\end{equation*}
\end{Theorem}
We now turn to estimating points in a fixed field of bounded height.
\begin{Theorem}\label{AsymptoticsForQ}
We have the following asymptotics as $N\to\infty $:
\begin{equation*}
\cardop\{ x\in\bbp^1(\bbq )\mid H(x)\le N\} =\frac{12}{\pi^2}N^2+O(N\log N).
\end{equation*}
\end{Theorem}
\begin{proof}
This is classical and can, up to reformulation into more elementary language, be found in \cite{Apostol76}, in the proof of Theorem 3.9.
\end{proof}
For other number fields, there is a similar estimate:
\begin{Theorem}[Schanuel]\label{AsymptoticsForK}
Let $K$ be a number field of degree $d_K>1$. Then for $N\to\infty$ we have
\begin{equation*}
\cardop\{ x\in\bbp^1(K)\mid H(x)\le N\}=c_K\cdot N^{2d_K}+O(N^{2d_K-1}),
\end{equation*}
with the constant
\begin{equation*}
c_K=\frac{2^{2r_1+r_2-1}(2\pi )^{r_2}}{\sqrt{|\Delta_K|}}\cdot\frac{\operatorname{Res}_{s=1}\zeta_K(s)}{\zeta_K(2)}=\frac{h_K\cdot R_K\cdot 2^{3r_1+r_2-1}\cdot (2\pi )^{2r_2}}{w_K\cdot |\Delta_K|\cdot \zeta_K(2)}.
\end{equation*}
Here, as usual, $r_1$ is the number of real places, $r_2$ the number of complex places, $\Delta_K$ the discriminant, $\zeta_K$ the Dedekind zeta function, $h_K$ the class number, $R_K$ the regulator and $w_K$ the number of roots of unity in $K$.
\end{Theorem}
\begin{proof}
This is a special case of the main result in \cite{Schanuel79}; the equality of the two expressions for $c_K$ follows from the class number formula. Note that Schanuel uses a different normalisation for the height, whence the different exponent.
\end{proof}
Note that for $K=\bbq$, the formula for $c_K$ gives $12/\pi^2$, as above; the only reason that we have to treat this case seperately is that the error term has a different shape. And, of course, Theorem \ref{AsymptoticsForQ} is much more elementary than Theorem \ref{AsymptoticsForK}.

The notion of height helps us to measure the ``size'' of a subset $A\subseteq\bbp^1(K)$.
\begin{Definition}
Let $K$ be a number field and $A\subseteq\bbp^1(K)$. Its \emph{lower height density} is the number
\begin{equation*}
\delta_h^{-}(A)=\liminf_{N\to\infty }\frac{\cardop\{ x\in A\mid H(x)\le N\} }{\cardop\{ x\in\bbp^1(K)\mid H(x)\le N\} }\in [0,1];
\end{equation*}
its \emph{upper height density} is the number
\begin{equation*}
\delta_h^{+}(A)=\limsup_{N\to\infty }\frac{\cardop\{ x\in A\mid H(x)\le N\} }{\cardop\{ x\in\bbp^1(K)\mid H(x)\le N\} }\in [0,1].
\end{equation*}
If these two are equal, we say that ``$A$ has a height density'' and call the quantity $\delta_h(A)=\delta_h^-(A)=\delta_h^+(A)$ the height density of $A$.
\end{Definition}
By Theorems \ref{AsymptoticsForQ} and \ref{AsymptoticsForK}, we see that $A$ has a height density if and only if the limit
\begin{equation*}
\lim_{N\to\infty }\frac{\cardop\{ x\in A\mid H(x)\le N\} }{N^2}
\end{equation*}
exists, and the height density is then this limit divided by the constant $c_K$.

We now give some examples for height density.
\begin{enumerate}
\item If $K$ is given as a subfield of $\bbr$, then the set of all positive $x\in K$ has height density $\frac{1}{2}$. This is because $H(x)=H(-x)$.
\item If $K$ is a number field of degree $d$, $\gamma\in\PGL_2(K)$ is a M\"{o}bius transformation and $A\subseteq\bbp^1(K)$ is any subset, then
$$\delta_h^-(\gamma (A))\ge\frac{\delta_h^-(A)}{(2H(\gamma ))^{2d}}\text{ and }\delta_h^+(\gamma (A))\le (2H(\gamma ))^{2d}\delta_h^+(A).$$
This follows from Proposition \ref{DistortionOfHeightByMoebius} together with the observation that the number of points of height below $N$ grows like $N^{2d}$. In particular if $A$ has nonzero lower height density, then so has $\gamma (A)$.
\item Combining the two previous examples, we see: if $K\subset\bbr$ is a number field and $a<b$, then the subset $K\cap [a,b]\subset\bbp^1(K)$ has positive lower height density, since there exists a M\"{o}bius transformation in $\PGL_2(K)$ which maps $[0,\infty )$ into $[a,b]$.
\item The set $\bbq_{>0}^{\text{even}}$ introduced before has positive height density in $\bbp^1(\bbq )$. This can be seen as follows. Let us estimate the number of pairs $(p,q)\in\bbn^2$ with $p,q$ coprime, $q\le p\le N$ and $p$ even. If we can show that this number is bounded below by some positive constant times $N^2$, we are done.

Now this number is equal to
\begin{equation*}
\sum_{\substack{1<p\le N\\ p\text{ even}}}\varphi (p)\ge\sum_{\substack{1<p\le N\\ p\text{ even}}}\varphi\left(\frac{p}{2}\right) =\sum_{n=1}^{\lfloor N/2\rfloor }\varphi (n)=\frac{3}{\pi^2}\cdot\left(\frac{N}{2}\right)^2+O(N\log N).
\end{equation*}
The first inequality is derived from the elementary inequality $\varphi (2n)\ge \varphi (n)$, and the final equality follows from \cite[Theorem 3.7]{Apostol76}.
\end{enumerate}

\section{Constraints on Injective Families}

In this final section we shall show that if an injective family consists only of maps of degree at least two, then its image in $\bbp^1(K)$ must have height density zero. So to get started, assume that $K$ is a number field and $(f_1,\ldots ,f_n)\in\mathrsfs{R}(K)^n$ is an injective family for some initial value $x_0\in\bbp^1(K)$, where $\deg f_i\ge 2$ for all $i$. Denote by $\Gamma$ the free monoid generated by the $f_i$ in $\mathrsfs{R}(K)$, and let $\lVert\gamma\rVert$ be the word norm on $\Gamma$. That is, for $\gamma =f_{i_1}f_{i_2}\cdots f_{i_r}$ set $\lVert\gamma\rVert =r$.

We prefer to work with logarithmic heights in this section. By Theorem \ref{TheoremSilverman}, we find a constant $c>0$ such that for all $1\le i\le n$ and all $x\in\bbp^1(K)$, the inequality
\begin{equation}\label{WachstumAusserhalbS}
h(f_i(x))\ge 2h(x)-c
\end{equation}
holds. By replacing $c$ with a larger constant if necessary, we may also assume that
$$c\ge 1.$$
Hence $\Gamma$ ``explodes'' heights outside the \emph{exceptional set}
\begin{equation*}
S=\{ x\in\bbp^1(K)\mid h(x)\le 2c\} .
\end{equation*}
By Theorem \ref{AsymptoticsForQ} or \ref{AsymptoticsForK}, depending on whether $K=\bbq$ or not, this is a finite set.
\begin{Lemma}\label{LemmaUeberWachstumAusserhalbS}
Under these assumptions, every element of $\Gamma$ takes the complement of $S$ to itself. In formul\ae :
\begin{equation}
\Gamma (\bbp^1(K)\smallsetminus S)\subseteq \bbp^1(K)\smallsetminus S.
\end{equation}
Furthermore, for any $x\in\bbp^1(K)\smallsetminus S$ and $\gamma\in\Gamma$ we have the inequality
\begin{equation}
h(\gamma (x))\ge\left(\frac{3}{2}\right)^{\lVert\gamma\rVert }\cdot h(x).
\end{equation}
\end{Lemma}
\begin{proof}
Let $x$ be in the complement of $S$, i.e. $h(x)>2c$. Then from (\ref{WachstumAusserhalbS}), we obtain
$$h(f_i(x))\ge 2h(x)-c>4c-c>2c.$$
In particular, $f_i(x)\notin S$. Since the $f_i$ generate $\Gamma$, this shows the first part.

The second inequality also needs only to be checked for $\gamma =f_i$ or, equivalently, $\lVert\gamma\rVert =1$. But using that $c<\frac{1}{2}h(x)$, we find that
$$h(f_i(x))\ge 2h(x)-c>2h(x)-\frac{1}{2}h(x)=\frac{3}{2}h(x),$$
which is just what is to be proved for $\lVert\gamma\rVert =1$.
\end{proof}
If one enlarges $S$ suitably, the estimate can of course be sharpened in such a way that the constant $\frac{3}{2}$ can be replaced by any $2-\varepsilon$ with $\varepsilon >0$.

Because the orbit map $\gamma\mapsto\gamma (x_0)$ is injective, it can hit $S$ only up to a finite word length. So there exists some $n_0\in\bbn$ with the property that whenever $\lVert\gamma\rVert \ge n_0$, then $\gamma (x_0)\notin S$ (and consequently $h(\gamma (x_0))>2c$).

\begin{Lemma}\label{LemmaOnExponentialGrowth}
Let $\gamma\in\Gamma$ with $\lVert\gamma\rVert >n_0$. Then
\begin{equation*}
h(\gamma (x_0))>\left(\frac{3}{2}\right)^{\lVert\gamma\rVert -n_0}.
\end{equation*}
\end{Lemma}
\begin{proof}
Set $N=\lVert\gamma\rVert -n_0$. Write $\gamma =\gamma_1\gamma_2$ with $\lVert\gamma_1\rVert =N$ and $\lVert\gamma_2\rVert =n_0$. Then
\begin{equation*}
h(\gamma (x_0))=h(\gamma_1(\gamma_2(x_0)))\ge\left(\frac{3}{2}\right)^{\lVert\gamma_1\rVert }\cdot h(\gamma_2(x_0))>\left(\frac{3}{2}\right)^N\cdot 2c>\left(\frac{3}{2}\right)^N.
\end{equation*}
The ``$\ge $'' sign is obtained from Lemma \ref{LemmaUeberWachstumAusserhalbS}, setting $x=\gamma_2(x_0)\notin S$ (by assumption on $\gamma_2$). The first ``$>$'' is justified again by the observation that $\gamma_2(x_0)\notin S$ and the definition of $S$. The second ``$>$'' sign finally is justified by $c\ge 1$ (remember we made it that way).
\end{proof}
\begin{Proposition}\label{PropositionOnLogarithmicGrowth}
Under the above assumptions, there exist constants $c'>0$ and $k\in\bbn$ such that for all sufficiently big positive reals $B$ one has
\begin{equation}
\cardop\{\gamma\in\Gamma\mid h(\gamma (x_0))\le B\}\le c'\cdot B^k.
\end{equation}
\end{Proposition}
\begin{proof}
Since $\Gamma$ is free on $r$ generators, we get that
\begin{equation*}
\cardop\{\gamma\in\Gamma\mid\lVert\gamma\rVert\le C\}=\sum_{\nu =0}^{\lfloor C\rfloor } r^{\nu }\le r^{C+1}
\end{equation*}
if $r\ge 2$; for $r=1$ we get the even simpler estimate $\lfloor C\rfloor +1$ that will also do the job. We assume from now on that $r\ge 2$ since the calculation for $r=1$ is even easier.

By Lemma \ref{LemmaOnExponentialGrowth}, we find
\begin{equation*}
\begin{split}
\cardop\{\gamma\in\Gamma\mid h(\gamma (x_0))\le B\}
&\le\cardop\{\gamma\in\Gamma\mid\left(\frac{3}{2}\right)^{\lVert\gamma\rVert -n_0}\le B\}\\
&=\cardop\{\gamma\in\Gamma\mid (\lVert\gamma\rVert -n_0)\log\frac{3}{2}\le\log B\}\\
&=\cardop\{\gamma\in\Gamma\mid \lVert\gamma\rVert\le n_0+\frac{\log B}{\log\frac{3}{2}}\}\\
&\le r^{n_0+\log B /\log\frac{3}{2}+1}=r^{n_0+1}\cdot B^{\log r/\log\frac{3}{2}},
\end{split}
\end{equation*}
so that setting $c'=r^{n_0+1}$ and $k=\lceil\log r/\log\frac{3}{2}\rceil$ will yield the desired estimate.
\end{proof}
\begin{Theorem}\label{LastTheorem}
Let $K$ be a number field and $(f_1,\ldots ,f_n)\in\mathrsfs{R}(K)^n$ an injective family for the initial value $x_0\in\bbp^1(K)$. Assume that $\deg f_i\ge 2$ for all $1\le i\le n$. Let $\Gamma\subset\mathrsfs{R}(K)$ be the submonoid generated by the $f_i$. Then the image $\Gamma (x_0)\subseteq\bbp^1(K)$ has height density zero.
\end{Theorem}
\begin{proof}
We translate the previous considerations back from statements about logarithmic heights into statements about heights. Since $H(x)\le N$ if and only if $h(x)\le\log N$, we see from Proposition \ref{PropositionOnLogarithmicGrowth} that there exists a positive integer $k$ with
\begin{equation*}
\cardop\{ x\in\Gamma (x_0)\mid H(x)\le N\} =O((\log N)^k).
\end{equation*}
Comparing this with Theorems \ref{AsymptoticsForQ} and \ref{AsymptoticsForK}, we see that $\Gamma (x_0)$ must have height density zero.
\end{proof}
We have seen before that in the case $K=\bbq$, for every $n\ge 2$ there exists an injective family whose orbit has positive height density and which consists of $n$ M\"{o}bius transformations. It is easy to see that we cannot get positive height density for a family consisting of just one M\"{o}bius transformation. Note, however, that Newman's map
\begin{equation*}
x\mapsto\frac{1}{1+\lfloor x\rfloor -\{ x\} },
\end{equation*}
being not terribly far apart from a M\"{o}bius transformation, gives an``injective family'' with just one element, whose orbit $\bbq_{>0}$ has height density $\frac{1}{2}$.

The last theorem tells us that we cannot get positive height density if we only work with maps of higher degree. So there remain two open questions: what about the mixed case, i.e. injective families consisting of both M\"{o}bius transformations and higher degree maps, and what about M\"{o}bius transformations in general number fields?

We conjecture that the condition ``$\deg f_i\ge 2$ for all $i$'' in Theorem \ref{LastTheorem} can be relaxed to the weaker condition ``$\deg f_i\ge 2$ for at least one $i$''. In other words, that if the orbit of an injective family has positive upper height density, then the family must consist entirely of M\"{o}bius transformations. Note that then the injectivity of the family would be a crucial condition since otherwise we could just add some higher degree maps to the Calkin-Wilf family. As to the second question, there might be interesting trees similar to the Calkin-Wilf tree already over quadratic number fields.

\end{document}